\documentclass[12pt,reqno]{article}
\usepackage{latexsym, amssymb, amscd, mathrsfs,geometry, amsmath, rotating,amsthm,color}

\newtheorem{thm}{Theorem}[section]
\newtheorem{lem}[thm]{Lemma}
\newtheorem{prop}[thm]{Proposition}

\newtheorem{Def}{Definition}[section]
\newtheorem{ex}{Example}[section]

\newtheorem{rem}{Remark}
\newcounter{Par}[subsection]

\DeclareMathOperator{\Supp}{Supp}
\def\o{\overline}
\def\a{\alpha}
\def\b{\beta}
\def\g{\gamma}

\def\la{\lambda}
\def\ve{\varepsilon}
\def\F{\mathfrak}
\def\r{\mathrm}

\def\C{\mathcal}
\def\CC{\mathbb{C}}
\def\ZZ{\mathbb{Z}}
\def\beqn{\begin{equation*}}
	\def\eeqn{\end{equation*}}
\def\beq{\begin{equation}}
	\def\eeq{\end{equation}}
\def\({\left(}
\def\){\right)}

\counterwithin{equation}{section}

\title{Quasi-Whittaker supermodules over Lie superalgebras}
\author{Wenting Gao, Baiying He, Shiyuan Liu, Jialin Xu}

\date{\today}

\begin{document}
	\maketitle
	
\begin{abstract}
	In this paper, we develop a general theory of quasi-Whittaker supermodules over Lie superalgebras induced from an arbitrary ideal. We determine the quasi-Whittaker vectors in universal supermodules, establish an irreducibility criterion, and classify several families of irreducible supermodules. The odd part produces a new irreducibility phenomenon absent from the Lie algebra setting. As applications, we determine all  irreducible quasi-Whittaker supermodules over   the $N=1$ super Schrödinger algebra and the $N=1$ $\frac{3}{2}$-conformal Galilei superalgebra, and over the complete spectrum-generating superalgebra in a special case.	
	\end{abstract}
	
	Keyword: quasi-Whittaker supermodule, Lie superalgebra, quasi-Whittaker vector

2020 MCS: 17B10, 17B20, 17B65

\section{Introduction}

Whittaker modules form an important class of non-weight modules in
representation theory. They first appeared for $\F{sl}_2$ in the 
work of Arnal and Pinczon \cite{AP}, and were subsequently
introduced systematically by Kostant for finite-dimensional complex
 semisimple Lie algebras \cite{K}. More precisely, let $\F{g}=\F{n}
 _{-}\oplus\F{h}\oplus\F{n}_{+}$ be a complex semisimple Lie algebra
 and let $\eta\colon\F{n}_{+}\rightarrow\CC$ be a Lie algebra
  homomorphism. A $\F{g}$-module is called a Whittaker module if 
 $x-\eta(x)$ acts locally nilpotently for every $x\in\F{n}_{+}$.

 Various classes of irreducible Whittaker modules in this setting were 
 investigated in \cite{B,K,M1,M2,MS}. Analogues and generalizations
 have since been studied for the Virasoro algebra and related
 algebras \cite{LPX1,LZ,LGZ,OW}, affine Kac--Moody algebras 
 \cite{ALZ,CGLW}, Lie algebras of polynomial vector fields \cite{ZL},
  classical Lie superalgebras \cite{C}, and quantum groups
  \cite{O,XZ,XGZ}.
A general framework based on Whittaker pairs was developed in \cite{BM}; see also \cite{MZ}. 

For the Schr\"odinger algebra, Cai, Cheng and Shen \cite{CCS} studied modules induced from the Heisenberg subalgebra and called them quasi-Whittaker modules in order to distinguish them from modules induced from Borel or parabolic subalgebras. Quasi-Whittaker modules were subsequently investigated for the Euclidean algebra $\F{e}(3)$, conformal Galilei algebras, the $(2+1)$-dimensional spacetime Schr\"odinger algebra, and the $n$-th Schr\"odinger algebra \cite{CC,CL,CSZ,CW1}; their categorical structure was further studied in \cite{JLZ}. 

The common features of these examples led Cheng et al. \cite{CGLZZ} to formulate a uniform theory for quasi-Whittaker modules over nonsemisimple Lie algebras induced from nonperfect ideals. In particular, they introduced the Whittaker annihilator, determined quasi-Whittaker vectors and the irreducibility of universal quasi-Whittaker modules, and classified several families of irreducible modules.

    Whittaker (super)modules for Lie superalgebras have also been extensively studied, including those for simple, classical, and quasi-reductive Lie superalgebras \cite{BCW,C,CC2}, the Neveu--Schwarz and Ramond algebras \cite{LPX1,LPX2,CYZ}, the $N=1$ and $N=2$ super-BMS$_3$ algebras \cite{LPXZ1,DGL,WGL}, the superconformal current algebra \cite{LPXZ2}, a subalgebra of the $N=2$ superconformal algebra \cite{JXZ}, and the $N=1$ super Schr\"odinger algebra \cite{WCM2}
    
Recently, Wang, Chen and Ma \cite{WCM} introduced and studied quasi-Whittaker supermodules for the $N=1$ super Schrödinger algebra. Their work inspired us to extend the general theory of quasi-Whittaker modules for Lie algebras developed in \cite{CGLZZ} to Lie superalgebras. This extension reveals a genuine structural difference: the odd part changes the irreducibility criterion for universal quasi-Whittaker supermodules $W(\phi)$, producing an additional irreducible case with no analogue in the Lie algebra setting.

More precisely, we determine all quasi-Whittaker vectors in $W(\phi)$ and establish a sharp irreducibility criterion. In particular, $W(\phi)$ may be irreducible even when $\F{g}^{\phi}\neq\F{p}$, a phenomenon with no Lie algebra analogue. We also characterize bland quasi-Whittaker supermodules and classify all irreducible quasi-Whittaker supermodules when both the even and odd components of $\F{g}^{\phi}/\F{p}$ are one-dimensional. Our method yields classifications for the $N=1$ super Schrödinger algebra, the $N=1$ $\frac{3}{2}$-conformal Galilei superalgebra, and, in a special case, the complete spectrum-generating superalgebra. The classification obtained for the $N=1$ super Schrödinger algebra differs from that stated in \cite{WCM}.

The paper is organized as follows. In Section~2, we introduce quasi-Whittaker supermodules over Lie superalgebras, define the Whittaker annihilator, and establish the preliminary results needed later. Section~3 contains the main results: we determine the quasi-Whittaker vectors in $W(\phi)$, give the irreducibility criterion for $W(\phi)$, classify bland supermodules, and treat the special case in which both homogeneous components of $\F{g}^{\phi}/\F{p}$ are one-dimensional. In Section~4, we apply the general theory to the three Lie superalgebras described above.

\section{Quasi-Whittaker supermodules}

We first introduce the notation used throughout the paper. For a well-ordered set $I$ (finite, countable, or uncountable), let
\beqn
\ZZ_+^I=\{\a\colon I\rightarrow \ZZ_+\mid \a(i)\neq 0~\text{only for finitely many $i\in I$}\}.
\eeqn 
For $\a\in\ZZ_+^I$,  write $\a_i=\a(i)$ for $i\in I$, and define 
\beqn
|\a|=\sum_{i\in I}\a_i.
\eeqn
For $i\in I$, let $\varepsilon_i\in\ZZ_+^I$ be defined by
\beqn
\varepsilon_{ij}=\begin{cases}
	1, \quad &\text{if $j= i$};\\
	0, \quad &\text{if $j\neq i$},
\end{cases}\quad \forall j\in I.
\eeqn

For $\a\in\ZZ_+^I$ and a symbol $x$, we use the notation
\beqn
x^\a=x_{i_1}^{\a_{i_1}}x_{i_2}^{\a_{i_2}}\cdots x_{i_n}^{\a_{i_n}}, 
\eeqn
where $i_1>i_2>\cdots>i_n$ are chosen such that 
\beqn
\a_k=0, \quad\forall k\in I\backslash \{i_1, i_2, \cdots, i_n\}.
\eeqn
Here, $x^{0}$ is understood to be $1$.

For $\a, ~\b\in\ZZ_+^I$, we write $\a>\b$ if either $|\a|>|\b|$, or $|\a|=|\b|$ and there exists $i\in I$ such that $\a_i>\b_i$ and  $\a_j=\b_j$ for each  $j>i$. This defines a total order on $\ZZ_+^I$.

Let $\a\in\ZZ_+^I$ with $|\a|>0$. The {\bf height} of $\a$, denoted by $\r{ht}(\a)$, is the smallest index $k\in I$ such that $\a_k\neq0$. Set
\beqn
\hat{\a}=\a-\varepsilon_{\r{ht}(\a)}.
\eeqn

\begin{lem}[\cite{CGLZZ}]\label{L2.1}
	Let $\a,\b\in\ZZ^I_+$ satisfy $|\a|,|\b|>0$ and $\a>\b$. Then $\hat{\a}\geq\hat{\b}$. Moreover, equality holds if and only if there exists $j<\r{ht}(\a)$ such that $\b=\hat{\a}+\varepsilon_j$.
\end{lem}

Throughout the paper, all subalgebras  a Lie superalgebra are understood to be $\ZZ_2$-graded. For a homogeneous element $x$ of a $\ZZ_2$-graded linear space, we write $|x|$ for its parity. Let $\F{g}$ be a Lie superalgebra and let $\F{p}$ be an ideal of $\F{g}$.

\begin{Def}
Let $\phi\colon\F{p}\rightarrow\CC$ be a Lie superalgebra homomorphism, where $\CC$ is regarded as a purely even abelian Lie superalgebra, and let $V$ be a $\F{g}$-supermodule.
\begin{itemize}
	\item[(1)] A vector $v\in V$ is called a {\bf quasi-Whittaker vector of type $\phi$} if $pv=\phi(p)v$ for each $p\in\F{p}$.
	\item[(2)] $V$ is said to be a {\bf quasi-Whittaker supermodule of type $\phi$} if $V$ is generated by a homogeneous quasi-Whittaker vector of type $\phi$.
\end{itemize}	
\end{Def}

\begin{rem}
If $\phi=0$, then the quasi-Whittaker condition reduces to the 
requirement that the cyclic vector be annihilated by $\F{p}$; thus
the notion includes every supermodule generated by such a homogeneous
 vector. To exclude this degenerate case, one usually assumes that 
 $\phi\neq0$. Since $\phi([\F{p},\F{p}])=0$, the existence of such a
 nonzero homomorphism forces $\F{p}$ to be nonperfect. Nevertheless,
 all the results below remain valid in the extreme case $\phi=0$.
\end{rem}

Motivated by \cite{CGLZZ}, we introduce the following definition.

\begin{Def}
The {\bf Whittaker annihilator} of $\phi$ is defined by
\beqn
\F{g}^{\phi}=\{y\in\F{g}\mid \phi([y, p])=0, ~\forall p\in\F{p}\}.
\eeqn	
\end{Def}

\begin{lem}
The Whittaker annihilator $\F{g}^{\phi}$ is a subalgebra of $\F{g}$ containing $\F{p}$.	
\end{lem}

\begin{proof}
It is plain that $\F{g}^\phi$ is a linear subspace of $\F{g}$. Let
$y=y_{\o{0}}+y_{\o{1}}\in\F{g}^\phi$. Since $\F{p}$ is an ideal,
$[y_{\o{0}},\F{p}_{\o{1}}]\subseteq\F{p}_{\o{1}}$, and hence
$\phi([y_{\o{0}},\F{p}_{\o{1}}])=0$. Moreover, for every
$p\in\F{p}_{\o{0}}$, we have $[y_{\o{1}},p]\in\F{p}_{\o{1}}$ and therefore
\beqn
\phi([y_{\o{0}},p])=\phi([y,p])-\phi([y_{\o{1}},p])=0.
\eeqn
Thus $y_{\o{0}}\in\F{g}^\phi$, and consequently
$y_{\o{1}}=y-y_{\o{0}}\in\F{g}^\phi$. Hence, $\F{g}^\phi$ is
$\ZZ_2$-graded.

Now let $y_1,y_2\in\F{g}^\phi$ and $p\in\F{p}$ be homogeneous. By the super Jacobi identity,
\beqn
\phi([p,[y_1,y_2]])=\phi([[p,y_1],y_2])+(-1)^{|p||y_1|}\phi([y_1,[p,y_2]])=0.
\eeqn
Therefore, $\F{g}^\phi$ is a subalgebra. Finally, since $\phi$ is a Lie superalgebra homomorphism into the abelian Lie superalgebra $\CC$, we have $\phi([\F{p},\F{p}])=0$. Thus $\F{p}\subseteq\F{g}^\phi$.
\end{proof}

For a $\F{g}$-supermodule $V$, set
\beq
V_{\phi}=\{v\in V\mid pv=\phi(p)v, ~\forall p\in \F{p}\}.
\eeq

\begin{lem}\label{L2.3}
Let $V$ be a $\F{g}$-supermodule. Then $V_{\phi}$ is a $\F{g}^{\phi}$-subsupermodule of $V$.
\end{lem}

\begin{proof}
It is clear that $V_\phi$ is a vector subspace of $V$.
Since $\phi$ is a Lie superalgebra homomorphism, $\phi(\F{p}_{\o{1}})=0$. We see that 
\beqn
\phi(p)=\phi(p_{\o{0}}), ~\forall p=p_{\o{0}}+p_{\o{1}}\in\F{p}.
\eeqn
 Assume that $v=v_{\o{0}}+v_{\o{1}}\in V_{\phi}$. Then 
\begin{align*}
&	p_{\o{0}}v_{\o{0}}+p_{\o{0}}v_{\o{1}}=p_{\o{0}}v=\phi(p_{\o{0}})v=\phi(p_{\o{0}})v_{\o{0}}+\phi(p_{\o{0}})v_{\o{1}}, ~\forall p_{\o{0}}\in\F{p}_{\o{0}};\\
&p_{\o{1}}v_{\o{0}}+p_{\o{1}}v_{\o{1}}=p_{\o{1}}v=0, ~\forall p_{\o{1}}\in \F{p}_{\o{1}}.
\end{align*}
It follows that 
\beqn
p_{\o{0}}v_{\o{0}}=\phi(p_{\o{0}})v_{\o{0}}, ~p_{\o{0}}v_{\o{1}}=\phi(p_{\o{0}})v_{\o{1}}, ~p_{\o{1}}v_{\o{0}}=p_{\o{1}}v_{\o{1}}=0, ~\forall p=p_{\o{0}}+p_{\o{1}}\in \F{p}.
\eeqn	
We obtain that
\beqn
pv_{\o{0}}=\phi(p)v_{\o{0}}, ~pv_{\o{1}}=\phi(p)v_{\o{1}}, ~\forall p\in\F{p}, 
\eeqn
Thus $v_{\o{0}},v_{\o{1}}\in V_{\phi}$, so $V_\phi$ is $\ZZ_2$-graded.

Let $v\in V_{\phi}$ and let $y\in\F{g}^{\phi}$ be homogeneous. Then
\begin{align*}
	&pyv=[p, y]v+(-1)^{|p||y|}ypv=\phi([p, y])v+(-1)^{|p||y|}y(\phi(p)v)\\
	=&0+(-1)^{|p||y|}\phi(p)yv=\phi(p)yv
\end{align*}
for each homogeneous $p\in\F{p}$.  Here, the last equality holds because $\phi(p)=0$ when $p$ is odd.
Therefore, $V_{\phi}$ is a $\F{g}^{\phi}$-subsupermodule of $V$.
\end{proof}

In what follows, we choose homogeneous bases compatible with the inclusions $\F{p}\subseteq\F{g}^\phi\subseteq\F{g}$, as follows:
\beq\label{F2.3}
\underbrace{x_i, ~i\in I;\quad \underbrace{y_j, ~j\in J;\quad\underbrace{p_k, ~k\in K, }_{\text{a homogeneous basis of }\F{p}}}_{\text{a homogeneous basis of }\F{g}^{\phi}}}_{\text{a homogeneous basis of} ~\F{g}}
\eeq
where $I$, $J$, and $K$ are well-ordered sets (or empty).
Set
\begin{align*}
	&\widetilde{\ZZ}_{+}^{I}=\{\a\in\ZZ_+^I\mid\a_i\in\{0,1\}\text{ if }|x_i|=\o{1}\},\\
	&\widetilde{\ZZ}_{+}^{J}=\{\b\in\ZZ_+^J\mid\b_j\in\{0,1\}\text{ if }|y_j|=\o{1}\}.
\end{align*}

Let $\C{U}(\F{g})$, $\C{U}(\F{g}^{\phi})$, and $\C{U}(\F{p})$ denote the universal enveloping algebras of $\F{g}$, $\F{g}^{\phi}$, and $\F{p}$, respectively. We have natural inclusions
\beq\label{F2.4}
\C{U}(\F{p})\subseteq \C{U}(\F{g}^{\phi})\subseteq \C{U}(\F{g}).
\eeq
By the PBW theorem, $\C{U}(\F{g})$ is a free right $\C{U}(\F{g}^{\phi})$-module with basis
\beqn
x^{\a}, \quad \a\in\widetilde{\ZZ}_+^I, 
\eeqn
and is a free right $\C{U}(\F{p})$-module with basis
\beqn
x^\a y^{\b}, \quad \a\in\widetilde{\ZZ}_+^I, ~\b\in\widetilde{\ZZ}_+^J.
\eeqn

For $v\in V_{\phi}$ and homogeneous elements $p\in\F{p}$ and $u\in\C{U}(\F{g})$, the equality $\phi(p)=0$ for odd $p$ gives
\beq\label{F2.5}
\(p-\phi(p)\)uv=[p,u]v.
\eeq
Here, for homogeneous $x,y\in\C{U}(\F{g})$,  
\beqn
[x, y]=xy-(-1)^{|x||y|}yx, \quad\forall x, y\in \C{U}(\F{g}),
\eeqn
and the bracket is extended bilinearly to $\C{U}(\F{g})$.

The following lemma is the Lie-superalgebra analogue of \cite[Lemma 2.4]{CGLZZ}. Its proof follows by tracking the Koszul signs in the argument given there and using \eqref{F2.5}; we omit the details.

\begin{lem}\label{L2.4}
Assume that $|I|>0$, and let $0\neq\a\in\widetilde{\ZZ}_+^I$ have height $k$. Let $V$ be a $\F{g}$-supermodule, let $p\in\F{p}$ be homogeneous, and let $v\in V_{\phi}$. Then
\beqn
\begin{aligned}
(p-\phi(p))x^\a v
&\equiv(-1)^{|p|\left|x^{\a-\a_k\varepsilon_k}\right|}
\a_k\phi([p,x_k])x^{\hat\a}v\\
&\quad \r{mod}~\r{Span}_{\CC}
\{x^\b v\mid \b<\hat\a,~\b_i\leq\a_i\text{ for all }i\in I\}.
\end{aligned}
\eeqn	
\end{lem}

\begin{thm}
Let $\F{g}$ be a Lie superalgebra and let $\F{p}$ be a solvable ideal of $\F{g}$ satisfying
$[\F{p}_{\o{1}},\F{p}_{\o{1}}]\subseteq[\F{p}_{\o{0}},\F{p}_{\o{0}}]$. Let $V$ be an irreducible $\F{g}$-supermodule. Then the following statements are equivalent:
\begin{itemize}
	\item[(1)] $V$ is locally finite as a $\F{p}$-supermodule;
	\item[(2)] $V$ is a quasi-Whittaker $\F{g}$-supermodule.
\end{itemize}
\end{thm}

\begin{proof}
$(1)\Rightarrow(2)$. Assume that $V$ is locally finite as a $\F{p}$-supermodule. Choose a nonzero homogeneous vector $v\in V$ and set $M=\C{U}(\F{p})v$. Then $M$ is a finite-dimensional $\F{p}$-supermodule. Since $\F{p}$ is solvable and satisfies $[\F{p}_{\o{1}},\F{p}_{\o{1}}]\subseteq[\F{p}_{\o{0}},\F{p}_{\o{0}}]$, \cite[Lemma 1.37]{CW} implies that $M$ contains a one-dimensional graded $\F{p}$-subsupermodule, say $\CC w$, with $w$ homogeneous. The action of $\F{p}$ on $\CC w$ defines a Lie superalgebra homomorphism $\phi\colon\F{p}\rightarrow\CC$ such that
\beqn
pw=\phi(p)w,\quad \forall p\in\F{p}.
\eeqn
Since $V$ is irreducible, $V=\C{U}(\F{g})w$. Thus $V$ is a quasi-Whittaker supermodule of type $\phi$.

$(2)\Rightarrow(1)$. Suppose that $V$ is a quasi-Whittaker supermodule of type $\phi$, generated by a homogeneous quasi-Whittaker vector $w$. For any $v\in V$, the PBW theorem gives
\beqn
v=\sum_{\a\in\widetilde{\ZZ}_{+}^{I},~\b\in\widetilde{\ZZ}_{+}^{J}}
c_{\a,\b}x^\a y^\b w,
\eeqn
where only finitely many coefficients $c_{\a,\b}\in\CC$ are nonzero. Lemmas \ref{L2.3} and \ref{L2.4} imply that
\beqn
\C{U}(\F{p})x^\a y^\b w
\subseteq\r{Span}_{\CC}\left\{x^{\a'}y^\b w\middle|
\a'\in\widetilde{\ZZ}_+^I,~\a'_i\leq\a_i~\text{for all }i\in I\right\}.
\eeqn
The space on the right is finite-dimensional. Since the PBW expansion of $v$ is finite, it follows that $\dim\C{U}(\F{p})v<\infty$. Thus $V$ is locally finite as a $\F{p}$-supermodule.
\end{proof}

Up to parity shift, it is enough in the sequel to consider quasi-Whittaker supermodule generated by an even cyclic quasi-Whittaker vector.
Let $\CC w_{\phi}$ be the one-dimensional $\F{p}$-supermodule defined by 
\beq
pw_{\phi}=\phi(p)w_{\phi}, ~\forall p\in\F{p}.
\eeq
The corresponding induced supermodule
\beq
W(\phi)=\r{Ind}_{\F{p}}^{\F{g}}\CC w_{\phi}
\eeq
is called the {\bf universal quasi-Whittaker supermodule} of type $\phi$. 

By the universal property of induction, if $V$ is a quasi-Whittaker $\F{g}$-supermodule of type $\phi$ generated by an even cyclic quasi-Whittaker vector $v$, then there is a unique $\F{g}$-supermodule homomorphism $\Phi\colon W(\phi)\rightarrow V$ such that $\Phi(w_{\phi})=v$. Since $v$ generates $V$, the homomorphism $\Phi$ is surjective.

We also use the symbol ``$>$" to denote the total order on $\widetilde{\ZZ}_+^I\times\widetilde{\ZZ}_+^J$ defined by 
\begin{align*}
(\a,\b)>(\a',\b')\Longleftrightarrow&  |\a|+|\b|>|\a'|+|\b'|;~\text{or~}  |\a|+|\b|=|\a'|+|\b'|~\text{but~} \a>\a';\\
&\text{or}~|\a|+|\b|=|\a'|+|\b'|~\text{and}~\a=\a'~\text{but~} \b>\b'.
\end{align*}

	\begin{prop}\label{P2.4}
	Each nonzero quasi-Whittaker vector in $W(\phi)$ is of type $\phi$.
\end{prop}

\begin{proof}
Let $w=uw_{\phi}\in W(\phi)$ be a nonzero quasi-Whittaker vector of type $\psi$. Write
\beqn
w=ax^{\a}y^{\b}w_{\phi}+\sum_{(\a',\b')<(\a,\b)}a_{\a',\b'}x^{\a'}y^{\b'}w_{\phi},
\eeqn
where $a\neq0$ and only finitely many $a_{\a',\b'}$ are nonzero. For $p\in\F{p}_{\o{0}}$, the coefficient of $x^{\a}y^{\b}w_{\phi}$ in $pw=\psi(p)w$ is $a\psi(p)$. On the other hand,
\beqn
pw=\([p,u]+up\)w_{\phi},
\eeqn
and the coefficient of the same leading PBW monomial is $a\phi(p)$. Hence $\psi(p)=\phi(p)$ for every $p\in\F{p}_{\o{0}}$. Since both $\psi$ and $\phi$ vanish on $\F{p}_{\o{1}}$, it follows that $\psi=\phi$.
		\end{proof}

\section{Main results}

In this section, we first characterize the quasi-Whittaker vectors in the universal supermodule and then determine when this supermodule is irreducible. We also classify the irreducible quasi-Whittaker supermodules whose spaces of quasi-Whittaker vectors are one-dimensional.

Throughout this section, let $\F{g}$ be a Lie superalgebra, let $\F{p}$ be an ideal of $\F{g}$, and let $\phi\colon\F{p}\rightarrow\CC$ be a Lie superalgebra homomorphism.

\subsection{Quasi-Whittaker vectors in $W(\phi)$}

We first determine all quasi-Whittaker vectors in $W(\phi)$. Set
\beqn
\widetilde{\CC}[y]=\r{Span}_{\CC}\{y^\b\mid\b\in\widetilde{\ZZ}_+^J\}\subseteq\C{U}(\F{g}).
\eeqn
Then we have 
\beq\label{F3.1}
\C{U}(\F{g}^\phi)w_{\phi}=\widetilde{\CC}[y]w_{\phi}. 
\eeq
Let $\CC[t]=\CC[t_j\mid j\in J]$ be the polynomial algebra over $\CC$ in the variables $t_j$, $j\in J$. When $J=\emptyset$, we set $\CC[t]=\CC$.
 Denote by
\beqn
\widetilde{\CC}[t]=\r{Span}_{\CC}\{t^\b\mid \b\in\widetilde{\ZZ}_+^J\}.
\eeqn
Then there exists a linear isomorphism
\beq
\hat{\cdot }\colon \widetilde{\CC}[y]\rightarrow \widetilde{\CC}[t], \quad y^\beta\mapsto t^\beta \quad (\beta\in\widetilde{\ZZ}_{+}^J).
\eeq
By the PBW theorem, $W(\phi)$ has basis
\beqn
x^{\a}y^{\b}w_{\phi}, \quad\a\in\widetilde{\ZZ}_{+}^{I}, ~\b\in\widetilde{\ZZ}_{+}^{J}.
\eeqn
Hence, every nonzero element $v\in W(\phi)$ can be written uniquely as
\beq
v=\sum_{\a\in\widetilde{\ZZ}_{+}^{I}}x^{\a}f_{\a}w_{\phi}
\eeq
with only finitely many nonzero $f_{\a}\in\widetilde{\CC}[y]$. Assume that $|I|>0$. Define the {\bf support} of $v$ by
\beqn
\Supp(v)=\lbrace\a\in\widetilde{\ZZ}_{+}^{I}\mid f_{\a}\neq0\rbrace,
\eeqn
and define its degree by
\beqn
\deg(v)=\max\Supp(v)
\eeqn
with respect to the total order ``$>$" on $\widetilde{\ZZ}_{+}^{I}$.

\begin{thm}\label{T3.1}
For the universal quasi-Whittaker supermodule $W(\phi)$, we have
$W(\phi)_\phi=\widetilde{\CC}[y]w_\phi$.
\end{thm}

\begin{proof}

If $I=\emptyset$, then $\F{g}^\phi=\F{g}$. In this case, we have
\beqn
W(\phi)=\C{U}(\F{g})w_{\phi}=\C{U}(\F{g}^\phi)w_{\phi}\subseteq W(\phi)_\phi.
\eeqn
Hence, we have 
\beqn
W(\phi)_{\phi}=W(\phi)=\C{U}(\F{g}^\phi)w_{\phi}=\widetilde{\CC}[y]w_\phi.
\eeqn

Now suppose that $I\neq\emptyset$. The inclusion $\widetilde{\CC}[y]w_\phi\subseteq W(\phi)_\phi$ follows from Lemma \ref{L2.3} and \eqref{F3.1}. If the reverse inclusion fails, then there exists a nonzero vector $v\in W(\phi)_\phi$ with $\deg(v)=\a\neq0$. Let $k=\r{ht}(\a)$ and
\beqn
\ell=\min\left\{\r{ht}(\g)\middle|\g\in\Supp(v)\right\}.
\eeqn
Then $k\geq\ell$, and the set
\beqn
T=\left\{j\in I\middle|k\geq j\geq\ell,~\g_j\neq0
\text{ for some }\g\in\Supp(v)\right\}
\eeqn
is finite. Write $T=\{k=i_1>i_2>\cdots>i_m=\ell\}$ and set
\beqn
\a^{(t)}=\a-\varepsilon_k+\varepsilon_{i_t},\quad t=1,\ldots,m.
\eeqn
Then $\a=\a^{(1)}>\a^{(2)}>\cdots>\a^{(m)}$ are adjacent elements of
\beqn
\Supp(v)\cup\{\a^{(1)},\a^{(2)},\ldots,\a^{(m)}\}.
\eeqn
Consequently,
\beqn
v=x^\a f_\a w_\phi+\sum_{t=2}^m x^{\a^{(t)}}f_{\a^{(t)}}w_\phi
+\text{lower-degree terms},
\eeqn
where $f_\a\neq0$. Since $\hat f_\a\neq0$, we may choose
$\mathbf{a}=(a_j)_{j\in J}$ with $a_j\in\CC$ such that
$\hat f_\a(\mathbf{a})\neq0$. The vector
\beqn
\a_k\hat f_\a(\mathbf{a})x_k
+\sum_{t=2}^m\hat f_{\a^{(t)}}(\mathbf{a})x_{i_t}
\eeqn
is nonzero and lies in the span of the chosen complement of $\F{g}^\phi$ in $\F{g}$. Hence, there exists a homogeneous $p\in\F{p}$ such that
\beqn
\a_k\phi([p,x_k])\hat f_\a(\mathbf{a})
+\sum_{t=2}^m\phi([p,x_{i_t}])\hat f_{\a^{(t)}}(\mathbf{a})\neq0.
\eeqn
It follows that
\beqn
\a_k\phi([p,x_k])f_\a
+\sum_{t=2}^m\phi([p,x_{i_t}])f_{\a^{(t)}}\neq0.
\eeqn
Lemma \ref{L2.4} therefore gives
\beqn
\begin{aligned}
(p-\phi(p))v
={}&(-1)^{|p|\left|x^{\a-\a_k\varepsilon_k}\right|}x^{\hat\a}
\left(\a_k\phi([p,x_k])f_\a
+\sum_{t=2}^m\phi([p,x_{i_t}])f_{\a^{(t)}}\right)w_\phi\\
&+\text{lower-degree terms}.
\end{aligned}
\eeqn
In particular, $(p-\phi(p))v\neq0$, contradicting $v\in W(\phi)_\phi$. Therefore, $W(\phi)_\phi=\widetilde{\CC}[y]w_\phi$.

\end{proof}

The preceding proof also yields the following proposition.
\begin{prop}\label{P3.2}
Each nonzero $\F{p}$-invariant subspace of $W(\phi)$ contains a nonzero quasi-Whittaker vector of type $\phi$.
\end{prop}

\begin{proof}
If $I$ is empty, then Theorem \ref{T3.1} gives $W(\phi)=W(\phi)_\phi$, and the assertion is immediate. 
Assume that $I\neq\emptyset$. Let $M$ be a nonzero $\F{p}$-invariant subspace of $W(\phi)$. Choose a nonzero vector $v\in M$ such that $|\deg(v)|$ is minimal among all nonzero elements of $M$. We claim that $v\in W(\phi)_\phi$. Suppose otherwise. Then $\a:=\deg(v)\neq0$, because the degree-zero subspace $\widetilde{\CC}[y]w_\phi$ is contained in $W(\phi)_\phi$. Using the notation $k,m,i_t$, and $f_{\a^{(t)}}$ from the proof of Theorem \ref{T3.1}, applied to $v$, we obtain a homogeneous element $p\in\F{p}$ such that
\beqn
\begin{aligned}
(p-\phi(p))v
={}&(-1)^{|p|\left|x^{\a-\a_k\varepsilon_k}\right|}x^{\hat\a}
\left(\a_k\phi([p,x_k])f_\a
+\sum_{t=2}^m\phi([p,x_{i_t}])f_{\a^{(t)}}\right)w_\phi\\
&+\text{lower-degree terms},
\end{aligned}
\eeqn
where $|\hat\a|=|\a|-1$. The vector $(p-\phi(p))v$ is nonzero and belongs to $M$, while
\beqn
|\deg((p-\phi(p))v)|=|\hat{\a}|=|\a|-1<|\deg(v)|,
\eeqn
This contradicts the choice of $v$. Therefore, $v\in W(\phi)_\phi$.
\end{proof}

\subsection{Irreducibility of $W(\phi)$}

We now discuss the irreducibility of $W(\phi)$.

\begin{lem}\label{L3.3}
If $W(\phi)$ is irreducible, then $\F{g}^\phi_{\o{0}}=\F{p}_{\o{0}}$.
\end{lem}

\begin{proof}
Suppose, to the contrary, that $\F{g}^\phi_{\o{0}}\neq\F{p}_{\o{0}}$. Then $J\neq\emptyset$, and we may choose the basis so that $y_{j_0}\in\F{g}^\phi_{\o{0}}$, where $j_0$ is the smallest element of $J$. The PBW theorem gives
	  \beqn
	  \C{U}(\F{g})y_{j_0}w_{\phi}=\r{Span}_{\CC}\{~x^\a y^{\b+\varepsilon_{j_0}}w_\phi\mid \a\in \widetilde{\ZZ}_+^I,~\b\in \widetilde{\ZZ}_+^J~\}.
	  \eeqn 
In particular, $w_\phi\notin\C{U}(\F{g})y_{j_0}w_{\phi}$. Thus $\C{U}(\F{g})y_{j_0}w_{\phi}$ is a nonzero proper subsupermodule of $W(\phi)$, a contradiction.
\end{proof}

\begin{thm}\label{T3.4}
The universal quasi-Whittaker supermodule $W(\phi)$ is irreducible if and only if one of the following conditions holds:
\begin{itemize}
\item[(1)] $\F{g}^\phi=\F{p}$;
\item[(2)] $\F{g}_{\o{0}}^\phi=\F{p}_{\o{0}}$, $\dim\F{g}_{\o{1}}^\phi/\F{p}_{\o{1}}=1$, and $\phi\bigl([\F{g}_{\o{1}}^\phi,\F{g}_{\o{1}}^\phi]\bigr)\neq0$.
\end{itemize}
\end{thm}

\begin{proof}
We first prove sufficiency. Let $M$ be a nonzero subsupermodule of $W(\phi)$. By Proposition \ref{P3.2}, $M\cap W(\phi)_\phi\neq0$. If (1) holds, Theorem \ref{T3.1} gives
\beqn
W(\phi)_\phi=\CC w_{\phi}.
\eeqn
and hence $w_\phi\in M$. If (2) holds, write $\F{g}^\phi_{\o{1}}=\F{p}_{\o{1}}\oplus\CC y$, where $\phi([y,y])\neq0$. By Theorem \ref{T3.1},
 \beqn
 W(\phi)_\phi=\CC w_{\phi}\oplus\CC yw_{\phi}.  
 \eeqn
 Since $M$ is $\ZZ_2$-graded, we have $w_{\phi}\in M$ or $yw_{\phi}\in M$. If $yw_{\phi}\in M$, we have
 \beqn
 w_{\phi}=\frac{2}{\phi([y,y])}y^2w_{\phi}\in M.
 \eeqn
In either case, $w_{\phi}\in M$, so $M=W(\phi)$. Thus $W(\phi)$ is irreducible.
 
We now prove necessity. Assume that $W(\phi)$ is irreducible. By Lemma \ref{L3.3}, $\F{g}^\phi_{\o{0}}=\F{p}_{\o{0}}$. Hence, $\phi$ induces a symmetric bilinear form
 \beqn
 (y_1+\F{p}_{\o{1}},y_2+\F{p}_{\o{1}})_\phi:=\phi([y_1,y_2]),\quad \forall y_1,~y_2\in \F{g}^{\phi}_{\o{1}}
 \eeqn
 on $\F{g}^{\phi}_{\o{1}}/\F{p}_{\o{1}}$.
This bilinear form is well defined: for $p_1,p_2\in\F{p}_{\o{1}}$ and $y_1,y_2\in\F{g}_{\o{1}}^\phi$,
\beqn
\phi\([y_1 + p_1,\, y_2 + p_2]\)
= \phi\([y_1,y_2]\) + \phi\([y_1,p_2]\) + \phi\([p_1,y_2]\) + \phi\([p_1,p_2]\)=\phi\([y_1,y_2]\).
\eeqn
The two middle terms vanish by the definition of $\F{g}^\phi$, while the last term vanishes because $\phi$ is a Lie superalgebra homomorphism and $\CC$ is abelian.

We claim that $\F{g}^{\phi}_{\o{1}}/\F{p}_{\o{1}}$ contains no nonzero isotropic vector with respect to $(\cdot,\cdot)_{\phi}$. Otherwise, $J$ is nonempty, and we may choose its smallest-indexed basis element $y_{j_0}\in\F{g}^\phi_{\o{1}}$ so that $\phi([y_{j_0},y_{j_0}])=0$. Then
\beqn
y_{j_0}^2w_{\phi}=\frac{1}{2}[y_{j_0},y_{j_0}]w_\phi=\frac{1}{2}\phi\([y_{j_0},y_{j_0}]\)w_\phi=0.
\eeqn 
Hence, for each $\a\in\widetilde{\ZZ}_+^I$ and $\b\in\widetilde{\ZZ}_+^J$,
 \beqn
 x^\a y^\b y_{j_0}w_{\phi}=\begin{cases}
0,\quad & \text{if $\b_{j_0}=1$};\\
 x^\a y^{\b+\varepsilon_{j_0}} w_{\phi},\quad& \text{if $\b_{j_0}=0$}
\end{cases}
 \eeqn 
  It follows that 
\beqn
\C{U}(\F{g})y_{j_0}w_{\phi}=\r{Span}_\CC\{~x^\a y^\b w_{\phi}\mid~\a\in \widetilde{\ZZ}_+^I~\text{and}~\b\in \widetilde{\ZZ}_+^J~\text{with}~\b_{j_0}= 1\}.
\eeqn
In particular, $w_{\phi}\notin\C{U}(\F{g})y_{j_0}w_{\phi}$. Hence $\C{U}(\F{g})y_{j_0}w_{\phi}$ is a nonzero proper subsupermodule of $W(\phi)$, contradicting irreducibility. Therefore, $\F{g}^{\phi}_{\o{1}}/\F{p}_{\o{1}}$ contains no nonzero isotropic vector. Every symmetric bilinear form on a complex linear space of dimension at least two has a nonzero isotropic vector; consequently,
\beqn
\dim \F{g}^{\phi}_{\o{1}}/\F{p}_{\o{1}}\le 1.
\eeqn
If $\dim\F{g}^{\phi}_{\o{1}}/\F{p}_{\o{1}}=0$, then $\F{g}_{\o{1}}^\phi=\F{p}_{\o{1}}$, and together with $\F{g}_{\o{0}}^\phi=\F{p}_{\o{0}}$ this gives (1). If the dimension is $1$, the absence of nonzero isotropic vectors implies $\phi([\F{g}_{\o{1}}^\phi,\F{g}_{\o{1}}^\phi])\neq0$, so (2) holds.
 \end{proof}
 
\begin{ex}
Let $\F{g}$ be the positive subalgebra of $\F{osp}(1|2)$. Then 
\beqn
\F{g}_{\o{0}}=\CC h\oplus\CC e,\quad \F{g}_{\o{1}}=\CC x
\eeqn
with nonzero brackets
\beqn
[h,e]=2e,\quad [h,x]=x,\quad [x,x]=2e.
\eeqn
The subspace $\F{p}=\CC e$ is an ideal of $\F{g}$. Let $\phi\colon\F{p}\rightarrow\CC$ be defined by $\phi(e)=1$. Then $\phi$ is a Lie superalgebra homomorphism, and a direct computation gives
\beqn
\F{g}^\phi=\CC e\oplus\CC x.
\eeqn	
Since $\phi([x,x])=2\neq0$, condition (2) of Theorem \ref{T3.4} holds. Hence $W(\phi)$ is irreducible.
\end{ex}

\subsection{Bland quasi-Whittaker supermodules}

\begin{Def}
Let $V$ be an irreducible quasi-Whittaker supermodule of type $\phi$. If $\dim V_\phi=1$, then $V$ is said to be {\bf bland}.	
\end{Def}

\begin{Def}
The Lie superalgebra homomorphism $\phi\colon\F{p}\rightarrow\CC$ is called {\bf extendable} if it extends to a Lie superalgebra homomorphism $\phi'\colon\F{g}^\phi\rightarrow\CC$.
\end{Def}

Assume that $\phi$ is extendable and $\phi'\colon \F{g}^\phi\rightarrow \CC$ is an extension of $\phi$. Define a one-dimensional $\F{g}^\phi$-supermodule $\CC v_{\phi'}$ by 
\beqn
yv_{\phi'}=\phi'(y)v_{\phi'},\quad \forall y\in\F{g}^\phi.
\eeqn
Set 
\beqn
V(\phi')=\r{Ind}_{\F{g}^\phi}^\F{g}\CC v_{\phi'}.
\eeqn 
The supermodule $V(\phi')$ is a quasi-Whittaker supermodule of type $\phi$ generated by the cyclic quasi-Whittaker vector $v_{\phi'}$. Hence, there exists a unique surjective $\F{g}$-supermodule homomorphism $\Phi\colon W(\phi)\rightarrow V(\phi')$ such that $\Phi(w_\phi)=v_{\phi'}$. Set
\beqn
W(\phi)_+=\r{Span}_\CC\{x^\a w_\phi\mid \a\in\widetilde{\ZZ}^I_+\}.
\eeqn
Then $W(\phi)_+$ is a $\F{p}$-invariant subspace of $W(\phi)$. Moreover, the restriction
\beqn
\Phi|_{W(\phi)_+}\colon W(\phi)_+\rightarrow V(\phi')
\eeqn
is a $\F{p}$-supermodule isomorphism.

\begin{lem}\label{L5}
	Under the assumptions above, the $\F{g}$-supermodule $V(\phi')$ is irreducible.
\end{lem}

\begin{proof}
Let $M$ be a nonzero subsupermodule of $V(\phi')$. Since $\Phi|_{W(\phi)_+}$ is an isomorphism, $\Phi^{-1}(M)\cap W(\phi)_+$ is a nonzero $\F{p}$-invariant subspace of $W(\phi)_+$. By Proposition \ref{P3.2}, it contains a nonzero quasi-Whittaker vector. Theorem \ref{T3.1} gives
\beqn
W(\phi)_+\cap W(\phi)_\phi=\CC w_\phi.
\eeqn
Thus $w_\phi\in\Phi^{-1}(M)\cap W(\phi)_+$, so $v_{\phi'}=\Phi(w_\phi)\in M$. Therefore, $M=V(\phi')$.
\end{proof}

\begin{thm}\label{T3.6}
There exists a bland quasi-Whittaker $\F{g}$-supermodule of type $\phi$ if and only if $\phi$ is extendable. In this case, the correspondence $\phi'\mapsto V(\phi')$ gives a bijection between the extensions of $\phi$ and the isomorphism classes, up to parity shift, of bland irreducible quasi-Whittaker $\F{g}$-supermodules of type $\phi$.
\end{thm}

\begin{proof}
Suppose first that $\phi'\colon\F{g}^\phi\rightarrow\CC$ is an extension of $\phi$. Lemma \ref{L5} shows that $V(\phi')$ is irreducible. Moreover, since $\Phi|_{W(\phi)_+}$ is a $\F{p}$-supermodule isomorphism, Theorem \ref{T3.1} gives
\beqn
V(\phi')_\phi=\Phi\bigl(W(\phi)_+\cap W(\phi)_\phi\bigr)=\CC v_{\phi'}.
\eeqn
Thus $V(\phi')$ is bland.

Conversely, let $V$ be a bland irreducible quasi-Whittaker supermodule generated by a homogeneous quasi-Whittaker vector $v$. Up to parity shift, assume that $v$ is even. By Lemma \ref{L2.3},
\beqn
\C{U}(\F{g}^\phi)v\subseteq V_\phi=\CC v.
\eeqn
The action of $\F{g}^\phi$ on $\CC v$ therefore defines a Lie superalgebra homomorphism $\phi'\colon\F{g}^\phi\rightarrow\CC$ extending $\phi$. The universal property of induction yields a surjective $\F{g}$-supermodule homomorphism $V(\phi')\rightarrow V$, which is an isomorphism by Lemma \ref{L5}.

Finally, suppose that $\phi'$ and $\phi''$ are extensions of $\phi$ and that $V(\phi')\cong V(\phi'')$. Any $\F{g}$-supermodule isomorphism maps $V(\phi')_\phi=\CC v_{\phi'}$ onto $V(\phi'')_\phi=\CC v_{\phi''}$. Comparing the actions of $\F{g}^\phi$ on these two lines gives $\phi'=\phi''$.
\end{proof}

\begin{thm}\label{T3.7}
Assume that $\F{g}^\phi_{\o{0}}=\F{p}_{\o{0}}\oplus\CC y_0$ and $\F{g}_{\o{1}}^\phi=\F{p}_{\o{1}}$. For each $\xi\in\CC$, the $\F{g}$-supermodule
\beq
M_\xi(\phi):=W(\phi)/\C{U}(\F{g})(y_0-\xi1)w_\phi
\eeq
is a bland quasi-Whittaker supermodule of type $\phi$. Moreover, every irreducible quasi-Whittaker $\F{g}$-supermodule of type $\phi$ is isomorphic, up to parity shift, to $M_\xi(\phi)$ for a unique $\xi\in\CC$.
 \end{thm}

\begin{proof}
Let $\xi\in\CC$. Define $\phi'\colon\F{g}^\phi\rightarrow\CC$ by extending $\phi$ and setting $\phi'(y_0)=\xi$. By the definition of $\F{g}^\phi$, this is a Lie superalgebra homomorphism. The construction of $V(\phi')$ yields a surjective $\F{g}$-supermodule homomorphism $\Psi\colon V(\phi')\rightarrow M_\xi(\phi)$ such that $\Psi(v_{\phi'})=\o{w_\phi}$. Since $V(\phi')$ is irreducible by Lemma \ref{L5}, $\Psi$ is an isomorphism. In particular, $M_\xi(\phi)$ is bland.

Let $V$ be an irreducible quasi-Whittaker supermodule of type $\phi$, generated by a homogeneous cyclic quasi-Whittaker vector $v$. Up to parity shift, assume that $v$ is even. Let $\Phi\colon W(\phi)\rightarrow V$ be the surjective $\F{g}$-supermodule homomorphism sending $w_\phi$ to $v$. By Theorem \ref{T3.4}, $W(\phi)$ is reducible, so $\Phi$ cannot be an isomorphism and $\r{Ker}~\Phi\neq0$.

By Proposition \ref{P3.2} and Theorem \ref{T3.1}, there exists a polynomial $f$ such that $f(y_0)w_\phi\in\r{Ker}~\Phi$. Since $w_\phi\notin\r{Ker}~\Phi$, the polynomial $f$ is nonconstant. Thus
\beqn
f(y_0)v=f(y_0)\Phi(w_\phi)=\Phi(f(y_0)w_\phi)=0.
\eeqn
It follows that $\CC[y_0]v$ is finite-dimensional. Since this space is $y_0$-invariant, it contains an eigenvector $v'$, say $y_0v'=\xi v'$ for some $\xi\in\CC$. By Lemma \ref{L2.3}, $v'$ is a quasi-Whittaker vector of type $\phi$, and irreducibility of $V$ implies that $v'$ is cyclic. Hence, the surjective homomorphism $W(\phi)\rightarrow V$ sending $w_\phi$ to $v'$ factors through a surjective homomorphism $M_\xi(\phi)\rightarrow V$. Since $M_\xi(\phi)$ is irreducible, this homomorphism is an isomorphism.

For uniqueness, suppose that $M_{\xi_1}(\phi)\cong M_{\xi_2}(\phi)$. Any such isomorphism maps the one-dimensional space $M_{\xi_1}(\phi)_\phi$ onto $M_{\xi_2}(\phi)_\phi$. Since $y_0$ acts on these spaces by the scalars $\xi_1$ and $\xi_2$, respectively, we obtain $\xi_1=\xi_2$.
\end{proof}

\subsection{A special case}

We next consider the special case in which
\beqn
\F{g}^\phi_{\o{0}}=\F{p}_{\o{0}}\oplus\CC y_0,\quad \F{g}^\phi_{\o{1}}=\F{p}_{\o{1}}\oplus\CC y_1,
\eeqn
and
\beqn
[y_1,y_1]=\la y_0+p_0,\quad [y_0,y_1]=\mu y_1+p_1,
\eeqn
where $p_0\in \F{p}_{\o{0}},~p_1\in\F{p}_{\o{1}}$ and $\la,\mu\in\CC$. 
We classify all irreducible quasi-Whittaker supermodules of type $\phi$ in this setting and apply the result in Section 4.

\begin{lem}\label{L3.8}
Under the assumptions above, let $V$ be an irreducible quasi-Whittaker $\F{g}$-supermodule generated by an even cyclic quasi-Whittaker vector $v$. Then there exists a nonconstant polynomial $f\in\CC[t]$ such that $f(y_0)v=0$.
\end{lem}

\begin{proof}
Let $\Phi\colon W(\phi)\rightarrow V$ be the surjective $\F{g}$-supermodule homomorphism such that $\Phi(w_\phi)=v$. Since $\F{g}^\phi_{\o{0}}\neq\F{p}_{\o{0}}$, Theorem \ref{T3.4} implies that $W(\phi)$ is reducible. Thus $\Phi$ cannot be an isomorphism, and $\r{Ker}~\Phi$ is a nonzero $\F{p}$-invariant subspace of $W(\phi)$.

By Proposition \ref{P3.2} and Theorem \ref{T3.1}, $\r{Ker}~\Phi$ contains a nonzero quasi-Whittaker vector in
\beqn
\CC[y_0]w_\phi\oplus y_1\CC[y_0]w_\phi.
\eeqn
Since $\r{Ker}~\Phi$ is graded, one of the homogeneous components of this vector is nonzero. Thus, for some nonzero $g\in\CC[t]$, either
\beqn
g(y_0)w_\phi\in\r{Ker}~\Phi
\qquad\text{or}\qquad
y_1g(y_0)w_\phi\in\r{Ker}~\Phi.
\eeqn

If the first alternative holds, then $g$ is nonconstant; otherwise, $w_\phi\in\r{Ker}~\Phi$, contradicting the surjectivity of $\Phi$. Taking $f=g$ proves the assertion in this case. We may therefore assume that no nonzero vector of the form $g(y_0)w_\phi$ lies in $\r{Ker}~\Phi$ and that $y_1g(y_0)w_\phi\in\r{Ker}~\Phi$. Then
\beqn
y_1^2g(y_0)w_\phi=\frac{1}{2}[y_1,y_1]g(y_0)w_\phi
=\frac{1}{2}\(\la y_0+\phi(p_0)\)g(y_0)w_\phi\in\r{Ker}~\Phi.
\eeqn
If $\la\neq0$ or $\phi(p_0)\neq0$, set
\beqn
f(t)=\frac{1}{2}\(\la t+\phi(p_0)\)g(t).
\eeqn
Then $f(y_0)w_\phi\in\r{Ker}~\Phi$ and $f\neq0$. Since $\r{Ker}~\Phi$ is proper, $f$ is nonconstant. Hence $f(y_0)v=0$.

It remains to consider the case $\la=\phi(p_0)=0$. Set $v'=g(y_0)v$. Since $g(y_0)w_\phi\notin\r{Ker}~\Phi$, we have $v'\neq0$. Moreover, $v'$ is an even quasi-Whittaker vector of type $\phi$ and is cyclic because $V$ is irreducible. Hence, there is a surjective $\F{g}$-supermodule homomorphism $\Phi'\colon W(\phi)\rightarrow V$ sending $w_\phi$ to $v'$. Since
\beqn
y_1v'=\Phi\bigl(y_1g(y_0)w_\phi\bigr)=0,
\eeqn
the map $\Phi'$ induces a surjective $\F{g}$-supermodule homomorphism
\beqn
\o{\Phi'}\colon W(\phi)/\C{U}(\F{g})y_1w_\phi\rightarrow V.
\eeqn
By the PBW theorem, the quotient has basis
\beqn
\o{x^\a y_0^k w_\phi},\quad \a\in\widetilde{\ZZ}_+^I,~k\in\ZZ_+.
\eeqn
The degree-reduction arguments used in Theorem \ref{T3.1} and Proposition \ref{P3.2} show that
\beqn
\(W(\phi)/\C{U}(\F{g})y_1w_\phi\)_\phi
=\r{Span}_\CC\{\o{y_0^kw_\phi}\mid k\in\ZZ_+\},
\eeqn
and that every nonzero $\F{p}$-invariant subspace of this quotient contains a nonzero quasi-Whittaker vector of type $\phi$.

Furthermore,
\beqn
y_1\o{y_0w_\phi}=\o{[y_1,y_0]w_\phi}+\o{y_0y_1w_\phi}=0.
\eeqn
Together with the PBW theorem, this shows that $\C{U}(\F{g})\o{y_0w_\phi}$ is a proper subsupermodule of $W(\phi)/\C{U}(\F{g})y_1w_\phi$. Hence $\o{\Phi'}$ is not an isomorphism, so $\r{Ker}~\o{\Phi'}$ is a nonzero $\F{p}$-invariant subspace of the quotient. It follows that there exists a nonzero polynomial $g_1\in\CC[t]$ such that $g_1(y_0)\o{w_\phi}\in\r{Ker}~\o{\Phi'}$. This polynomial is nonconstant because $\o{\Phi'}$ is surjective. Therefore,
\beqn
g_1(y_0)g(y_0)v=g_1(y_0)v'=0.
\eeqn
Taking $f=g_1g$ completes the proof.
\end{proof}

For $\xi\in\CC$, set 
\beq
\widetilde{L}_\xi(\phi)=W(\phi)/\C{U}(\F{g})(y_0-\xi1)w_\phi.
\eeq
By the PBW theorem, $\widetilde{L}_\xi(\phi)$ has a basis 
\beqn
\o{x^\a w_\phi},\quad \o{x^\a y_1w_\phi},\quad \a\in\widetilde{\ZZ}_+^I.
\eeqn
Moreover, we have
\begin{align}
&y_0\o{y_1w_\phi}=(\xi+\mu)\o{y_1w_\phi},\label{F3.5}\\ 
&y_1\o{y_1w_\phi}=\frac{1}{2}\(\la\xi+\phi(p_0)\)\o{w_\phi}.\label{F3.6}
\end{align}
The degree-reduction arguments used in Theorem \ref{T3.1} and Proposition \ref{P3.2} show that
\beq
\widetilde{L}_\xi(\phi)_\phi=\CC \o{w_\phi}\oplus\CC \o{y_1w_\phi},
\eeq
and that each nonzero $\F{p}$-invariant subspace of $\widetilde{L}_\xi(\phi)$ contains a nonzero quasi-Whittaker vector of type $\phi$.

\begin{lem}\label{L3.9}
(1) If $\la\xi+\phi(p_0)\neq 0$, then $\widetilde{L}_\xi(\phi)$ is irreducible.

(2) If $\la\xi+\phi(p_0)=0$, then $\widetilde{L}_\xi(\phi)$ has a unique nonzero proper subsupermodule $\C{U}(\F{g})\o{y_1w_\phi}$.
\end{lem}

\begin{proof}
(1) Let $M$ be a nonzero subsupermodule of $\widetilde{L}_\xi(\phi)$. Then $M_\phi\neq0$. Since $M$ is $\ZZ_2$-graded, either $\o{w_\phi}\in M$ or $\o{y_1w_\phi}\in M$. In the latter case, \eqref{F3.6} gives
\beqn
\o{w_\phi}=\frac{2}{\la\xi+\phi(p_0)}y_1\o{y_1w_\phi}\in M.
\eeqn	
Thus $\o{w_\phi}\in M$ in either case, so $M=\widetilde{L}_\xi(\phi)$. Therefore, $\widetilde{L}_\xi(\phi)$ is irreducible.

(2) By the PBW theorem and \eqref{F3.5}--\eqref{F3.6},
\beqn
\C{U}(\F{g})\o{y_1w_\phi}=\r{Span}_\CC\{\o{x^\a y_1w_\phi}\mid \a\in\widetilde{\ZZ}_+^I\}.
\eeqn
Thus $\C{U}(\F{g})\o{y_1w_\phi}$ is proper, and $\widetilde{L}_\xi(\phi)/\C{U}(\F{g})\o{y_1w_\phi}$ has basis
\beqn
\o{x^\a w_\phi}+\C{U}(\F{g})\o{y_1w_\phi},\quad \a\in\widetilde{\ZZ}_+^I.
\eeqn
The degree-reduction arguments used in Theorem \ref{T3.1} and Proposition \ref{P3.2} show that this quotient is irreducible. Hence, $\C{U}(\F{g})\o{y_1w_\phi}$ is a maximal proper subsupermodule of $\widetilde{L}_\xi(\phi)$.

Let $M$ be any nonzero proper subsupermodule of $\widetilde{L}_\xi(\phi)$. It contains a nonzero quasi-Whittaker vector of type $\phi$. Since $M$ is proper and $\ZZ_2$-graded, we must have $\o{y_1w_\phi}\in M$. Hence $\C{U}(\F{g})\o{y_1w_\phi}\subseteq M$, and maximality forces equality. Therefore, $\C{U}(\F{g})\o{y_1w_\phi}$ is the unique nonzero proper subsupermodule of $\widetilde{L}_\xi(\phi)$.
\end{proof}

For $\xi\in\CC$, set 
\beq
L_\xi(\phi)=\begin{cases}
	\widetilde{L}_\xi(\phi),\quad &\text{if~}\la\xi+\phi(p_0)\neq0;\\
	\widetilde{L}_\xi(\phi)/\C{U}(\F{g})\o{y_1w_\phi},\quad & \text{if~}\la\xi+\phi(p_0)=0.
\end{cases}
\eeq
Then $L_\xi(\phi)$ is an irreducible quasi-Whittaker supermodule of type $\phi$. Moreover, it is bland if and only if $\la\xi+\phi(p_0)=0$. We also have
\beq\label{F3}
L_\xi(\phi)\cong W(\phi)\big/\(\C{U}(\F{g})(y_0-\xi1)w_\phi+\delta_{0,\la\xi+\phi(p_0)}\C{U}(\F{g})y_1w_\phi\).
\eeq

\begin{thm}\label{T3.10}
Assume that
\beqn
\F{g}^\phi_{\o{0}}=\F{p}_{\o{0}}\oplus\CC y_0,\quad \F{g}^\phi_{\o{1}}=\F{p}_{\o{1}}\oplus\CC y_1.
\eeqn
and
\beqn
[y_1,y_1]=\la y_0+p_0,\quad [y_0,y_1]=\mu y_1+p_1,
\eeqn
where $p_0\in\F{p}_{\o{0}}$ and $p_1\in\F{p}_{\o{1}}$. Then every irreducible quasi-Whittaker $\F{g}$-supermodule of type $\phi$ is isomorphic, up to parity shift, to $L_\xi(\phi)$ for a unique $\xi\in\CC$.
\end{thm}

\begin{proof}
Let $V$ be an irreducible quasi-Whittaker supermodule of type $\phi$, generated by an even cyclic quasi-Whittaker vector $v$. By Lemma \ref{L3.8}, there exists a nonconstant polynomial $f\in\CC[t]$ such that $f(y_0)v=0$. Hence $\dim\CC[y_0]v<\infty$. Since $\CC[y_0]v$ is $y_0$-invariant, it contains an eigenvector $v_0$; thus, for some $\xi\in\CC$,
\beqn
y_0v_0=\xi v_0.
\eeqn
Since $y_0\in\F{g}^\phi$, Lemma \ref{L2.3} implies that $\CC[y_0]v\subseteq V_\phi$. Therefore, $v_0$ is an even quasi-Whittaker vector of type $\phi$. Since $V$ is irreducible and $v_0\neq0$, the vector $v_0$ is cyclic. The resulting surjective homomorphism $\Phi_0\colon W(\phi)\rightarrow V$, defined by $\Phi_0(w_\phi)=v_0$, induces a surjective $\F{g}$-supermodule homomorphism
\beqn
\o{\Phi}_0\colon\widetilde{L}_\xi(\phi)\rightarrow V.
\eeqn

If $\la\xi+\phi(p_0)\neq0$, Lemma \ref{L3.9}(1) implies that $\o{\Phi}_0$ is an isomorphism.

If $\la\xi+\phi(p_0)=0$, then $\widetilde{L}_\xi(\phi)$ is reducible by Lemma \ref{L3.9}(2). Since $V$ is irreducible, $\r{Ker}~\o{\Phi}_0$ is a nonzero proper subsupermodule of $\widetilde{L}_\xi(\phi)$. Again by Lemma \ref{L3.9}(2),
\beqn
\r{Ker}~\o{\Phi}_0=\C{U}(\F{g})\o{y_1w_\phi}.
\eeqn
Therefore, $\o{\Phi}_0$ induces an isomorphism $L_\xi(\phi)\cong V$.

Finally, suppose that $L_{\xi_1}(\phi)\cong L_{\xi_2}(\phi)$. Such an isomorphism preserves the even part of the space of quasi-Whittaker vectors, which is the line spanned by the image of $w_\phi$. Since $y_0$ acts on this line by $\xi_i$ in $L_{\xi_i}(\phi)$, we obtain $\xi_1=\xi_2$.
\end{proof}

\begin{rem}
	If $[y_1,y_1]\notin \F{p}_{\o{0}}$, then one can choose $y_0=[y_1,y_1]$ up to a nonzero scalar multiple. In this situation, $\phi(p_0)=0$, and therefore $L_\xi(\phi)$ is given by
	\beq\label{F10}
	L_\xi(\phi)= W(\phi)\big/\(\C{U}(\F{g})(y_0-\xi1)w_\phi+\delta_{0,\xi}\C{U}(\F{g})y_1w_\phi\).
	\eeq
\end{rem}

\section{Applications}

In this section, we apply the general results of Section 3 to three Lie superalgebras arising in mathematical physics: the $N=1$ super Schrödinger algebra, the $N=1$ $\frac{3}{2}$-conformal Galilei superalgebra, and the complete spectrum-generating superalgebra. For each algebra, we choose a natural graded ideal $\F{p}$ and determine the corresponding Whittaker annihilator $\F{g}^\phi$. We then use these computations to determine whether the universal quasi-Whittaker supermodule $W(\phi)$ is irreducible and, in the reducible cases, to classify all irreducible quasi-Whittaker supermodules of type $\phi$, up to parity shift.

\subsection{The $N=1$ super Schrödinger algebra}

The Schr\"odinger algebra describes the spacetime symmetries of the free nonrelativistic Schr\"odinger equation, while its supersymmetric extension incorporates fermionic symmetries. Thus, studying the representations of $\C{S}$ provides an algebraic framework for understanding symmetry sectors in nonrelativistic supersymmetric models.

Following \cite{WCM}, let $\C{S}$ denote the 9-dimensional complex $N=1$ super Schrödinger algebra with
\beqn
\C{S}_{\bar{0}} = \r{Span}_{\mathbb{C}}\{e, f, h, p, q, z\}, \quad \C{S}_{\bar{1}} = \r{Span}_{\mathbb{C}}\{E, F, G\},
\eeqn  
whose nonzero brackets are
\begin{alignat*}{4}
&[h,e] = 2e,  &\qquad &[h,f] = -2f,  &\qquad &[h,p] = p,  &\qquad &[h,q] = -q, \\
&[e,f] = h,   &\qquad &[p,q] = z,    &\qquad &[e,q] = p,  &\qquad &[p,f] = -q, \\
&[h,E] = E,   &\qquad &[h,F] = -F,   &\qquad &[e,F] = -E, &\qquad &[f,E] = -F, \\
&[p,F] = G,   &\qquad &[q,E] = G,    &\qquad &[E,E] = 2e, &\qquad &[F,F] = -2f, \\
&[G,G] = z,   &\qquad &[E,F] = h,    &\qquad &[E,G] = -p, &\qquad &[F,G] = q.
\end{alignat*}
Take 
\beqn
\F{p} = \r{Span}_{\mathbb{C}}\{p, q, z, G\}.
\eeqn
Then $\F{p}$ is an ideal of $\C{S}$. 

Let $\phi\colon \F{p}\rightarrow \CC$ 
be a Lie superalgebra homomorphism. Then
$\phi(z)=\phi(G)=0$, and $\phi$ is determined by $\phi(p)$ and $\phi(q)$.

\begin{prop}
Assume that $\phi\neq 0$. Then
\beqn
\C{S}^{\phi}_{\overline 0}=\F{p}_{\overline 0}\oplus \mathbb{C}y_0,\quad \C{S}^{\phi}_{\overline 1}=\F{p}_{\overline 1}\oplus \mathbb{C}y_1
\eeqn
where
\beqn
y_0=\phi(q)^2e+\phi(p)\phi(q)h-\phi(p)^2f,\quad y_1=\phi(q)E+\phi(p)F
\eeqn
and $[y_1,y_1]=2y_0$.
\end{prop}

\begin{proof}
	By the definition of $\C{S}^\phi$, an element  $y=\a_1 e+ \a_2h+\a_3f$ belongs to $\C{S}_{\o{0}}^\phi$ if and only if 
	\beqn
	\phi([y,p])=\phi([y,q])=0,
	\eeqn
	i.e., $\a_1,\a_2,\a_3$ satisfy the system of linear equations
	\beqn
	\left\{\begin{array}{c}
\phi(p)\a_2+\phi(q)\a_3 =0,\\
\phi(p)\a_1-\phi(q)\a_2=0.
\end{array}
\right.
	\eeqn
The coefficient matrix
\beqn
\begin{pmatrix}
	0 & \phi(p) & \phi(q)\\
	\phi(p) & -\phi(q) & 0
\end{pmatrix}
\eeqn
has rank $2$, and $(\phi(q)^2,\phi(p)\phi(q),-\phi(p)^2)^\r{T}$ is a nonzero solution. Set
\beqn
y_0=\phi(q)^2e+\phi(p)\phi(q)h-\phi(p)^2f.
\eeqn
Then $\C{S}_{\o{0}}^\phi=\F{p}_{\o{0}}\oplus \CC y_0$. 

Similarly, an element $y=\b_1E+\b_2F$ belongs to $\C{S}_{\o{1}}^\phi$ if and only if 
\beqn
\phi([y,G])=-\phi(p)\b_1 +\b_2 \phi(q)=0.
\eeqn
The solution space of this equation is spanned by $(\phi(q),\phi(p))^\r{T}$. Set
\beqn
y_1=\phi(q)E+\phi(p)F.
\eeqn
Then $\C{S}_{\o{1}}^\phi=\F{p}_{\o{1}}\oplus \CC y_1$. 

A direct computation gives $[y_1,y_1]=2y_0$.
\end{proof}

Theorem \ref{T3.10} now gives the following classification.

\begin{thm}
Assume that $\phi\colon\F{p}\rightarrow\CC$ is a nonzero Lie superalgebra homomorphism. Let
\beqn
y_0=\phi(q)^2e+\phi(p)\phi(q)h-\phi(p)^2f,\qquad y_1=\phi(q)E+\phi(p)F.
\eeqn
Then every irreducible quasi-Whittaker $\C{S}$-supermodule of type $\phi$ is isomorphic, up to parity shift, to the supermodule 
\beqn
L_\xi(\phi):= W(\phi)\big/\(\C{U}(\C{S})(y_0-\xi1)w_\phi+\delta_{0,\xi}\C{U}(\C{S})y_1w_\phi\)
\eeqn 
 for a unique $\xi\in\CC$.
\end{thm}
 
\subsection{The $N=1$ $\frac{3}{2}$-conformal Galilei superalgebra}

The $\ell$-conformal Galilei algebras \cite{GM} extend nonrelativistic conformal symmetry by incorporating acceleration generators, and their realizations for $\ell>\frac{1}{2}$ are closely related to higher-derivative mechanics. Their supersymmetric extensions therefore provide natural algebraic models for higher-derivative supermechanical systems, motivating a systematic study of their representations.

Following \cite{GM}, let $\F{g}=\F{g}_{\o{0}}\oplus\F{g}_{\o{1}}$ be the 12-dimensional $N=1$ $\frac{3}{2}$-conformal Galilei superalgebra with
\beqn
\F{g}_{\bar{0}} =\r{Span}_{\CC}\left\{ L_{-1}, L_0, L_1, P_{-\frac{3}{2}}, P_{-\frac{1}{2}}, P_{\frac{1}{2}}, P_{\frac{3}{2}} \right\},\qquad
\F{g}_{\bar{1}} =\r{Span}_{\CC}\left\{ Q_{-\frac{1}{2}}, Q_{\frac{1}{2}}, X_{-1}, X_0, X_1 \right\},
\eeqn
whose nonzero brackets are
\begin{alignat*}{3}
&[L_m, L_n] = (m-n)L_{m+n},
&\qquad &[L_n, Q_r] = \left(\frac{n}{2} - r\right) Q_{n+r},
&\qquad &[Q_r, Q_s] = 2L_{r+s},
\\
&[L_n, P_j] = \left(\frac{3}{2}n - j\right) P_{n+j},
&\qquad &[L_n, X_k] = (n-k)X_{n+k},
&\qquad &[Q_{r},P_{j}]=(3r-j)X_{r+j},
\\
&[Q_{r},X_{k}]=-P_{r+k},
\end{alignat*}
where 
\beqn
m,n\in \left\{ -1,0,1 \right\}, \quad r,s\in \left\{ -\dfrac{1}{2}, \dfrac{1}{2}\right\}, \quad j \in\left\{ -\dfrac{3}{2},-\dfrac{1}{2}, \dfrac{1}{2},\dfrac{3}{2}\right\}, \quad k\in \left\{ -1,0,1 \right\}.
\eeqn
Take
\beqn
\F{p}=\r{Span}_{\CC}\left\{ P_{-\frac{3}{2}},P_{-\frac{1}{2}},P_{\frac{1}{2}},P_{\frac{3}{2}},X_{-1},X_0,X_1 \right\}.
\eeqn
Then $\F{p}$ is an ideal of $\F{g}$.

Let $\phi\colon\F{p}\rightarrow\CC$ be a nonzero Lie superalgebra homomorphism. Then
\beqn
\phi(X_{-1})=\phi(X_0)=\phi(X_1)=0,
\eeqn
and $\phi$ is uniquely determined by
\beqn
\begin{aligned}
a&:=\phi\left(P_{-\frac{3}{2}}\right), & b&:=\phi\left(P_{-\frac{1}{2}}\right),&
c&:=\phi\left(P_{\frac{1}{2}}\right), & d&:=\phi\left(P_{\frac{3}{2}}\right).
\end{aligned}
\eeqn
By the definition of $\F{g}^\phi$, an element $y=\a_1 L_{-1}+\a_2 L_0+\a_3 L_1$ belongs to $\F{g}_{\o{0}}^\phi$ if and only if
\beqn
\phi\([y,P_j]\)=0,\quad \forall j\in\left\{ -\dfrac{3}{2},-\dfrac{1}{2}, \dfrac{1}{2},\dfrac{3}{2}\right\},
\eeqn
i.e., the coefficients $\a_1,\a_2,\a_3$ satisfy the linear system of equations
\beq\label{F4.1}
\begin{cases}
a\alpha_2+2b\alpha_3=0,\\
-2a\alpha_1+b\alpha_2+4c\alpha_3=0,\\
-4b\alpha_1-c\alpha_2+2d\alpha_3=0,\\
2c\alpha_1+d\alpha_2=0,
\end{cases}
\eeq
whose coefficient matrix is
\beqn
A_\phi=
\begin{pmatrix}
0 & a & 2b\\
-2a & b & 4c\\
-4b & -c & 2d\\
2c & d & 0
\end{pmatrix}.
\eeqn
Similarly, an element $y=\beta_1 Q_{-\frac12}+\beta_2 Q_{\frac12}$ belongs to $\F{g}^\phi_{\o{1}}$ if and only if 
\beqn
\phi\([y,X_i]\)=0,\quad \forall i\in \{-1,~0,~1\},
\eeqn
i.e., the coefficients satisfy the linear system
\beq\label{F4.2}
\begin{cases}
a\beta_1+b\beta_2=0,\\
b\beta_1+c\beta_2=0,\\
c\beta_1+d\beta_2=0,
\end{cases}
\eeq
whose coefficient matrix is
\beqn
B_\phi=
\begin{pmatrix}
a & b\\
b & c\\
c & d
\end{pmatrix}.
\eeqn

Since $\phi\neq0$, we have $\r{rank}A_\phi\geq2$ and $\r{rank}B_\phi\geq1$. Up to nonzero scalar multiples, the maximal minors of $A_\phi$ are $F_1,F_2,F_3,F_4$, and those of $B_\phi$ are $D_1,D_2,D_3$, where
\beqn
\left\{\begin{array}{l}
	F_1=a^2d-3abc+2b^3,\\
	F_2=abd-2ac^2+b^2c,\\
	F_3=acd-2b^2d+bc^2,\\
	F_4=ad^2-3bcd+2c^3,
\end{array}
\right.\quad
\left\{\begin{array}{l}
	D_1=ac-b^2,\\
	D_2=ad-bc,\\
	D_3=bd-c^2.
\end{array}
\right.
\eeqn
Hence, we have
\beq\label{F4.3}
\left\{\begin{array}{l}
	F_1=aD_2-2bD_1,\\
F_2=bD_2-2cD_1,\\
F_3=cD_2-2bD_3,\\
F_4=dD_2-2cD_3.
\end{array}
\right.
\eeq

\begin{lem}
Assume that $\phi\neq 0$. Then $\r{rank}A_\phi=2$ if and only if $\r{rank}B_\phi=1$.
\end{lem}

\begin{proof}
	Assume that $\r{rank}B_\phi=1$. Then \eqref{F4.3} gives 
	\beqn
	F_1=F_2=F_3=F_4=0.
	\eeqn
Thus $\r{rank}A_\phi<3$, and hence $\r{rank}A_\phi=2$.
	
	Suppose that $\r{rank}A_\phi=2$. Then $F_1=F_2=F_3=F_4=0.$
	Hence,
\begin{align*}
&2D_1^2=2(ac-b^2)D_1=bF_1-aF_2=0,\\
&2D_3^2=2(bd-c^2)D_3=cF_4-dF_3=0,
\end{align*}
Thus $D_1=D_3=0$. Substituting these equalities into \eqref{F4.3}, we obtain
\beqn
aD_2 = bD_2 = cD_2 = dD_2 = 0.
\eeqn
Since $\phi\neq0$, it follows that $D_2=0$. Therefore, $\r{rank}B_\phi<2$, and hence $\r{rank}B_\phi=1$.
\end{proof}

\begin{thm}
With the notation above, let $\phi\colon\F{p}\rightarrow\CC$ be a nonzero Lie superalgebra homomorphism.
\begin{itemize}
\item[(1)] If $\r{rank}B_\phi=2$, then $W(\phi)$ is irreducible.
\item[(2)] If $\r{rank}B_\phi=1$, let
\beqn
y_0=\b_1^2L_{-1}+2\b_1\b_2L_0+\b_2^2L_1,\qquad
y_1=\b_1Q_{-\frac12}+\b_2Q_{\frac12},
\eeqn
where $(\b_1,\b_2)^\r{T}$ is a nonzero solution of \eqref{F4.2}. Then
\beqn
\F{g}_{\o{0}}^\phi=\F{p}_{\o{0}}\oplus\CC y_0,\qquad
\F{g}_{\o{1}}^\phi=\F{p}_{\o{1}}\oplus\CC y_1,
\eeqn
and $[y_1,y_1]=2y_0$. Moreover, every irreducible quasi-Whittaker $\F{g}$-supermodule of type $\phi$ is isomorphic, up to parity shift, to the supermodule 
\beqn
L_\xi(\phi):= W(\phi)\big/\(\C{U}(\F{g})(y_0-\xi1)w_\phi+\delta_{0,\xi}\C{U}(\F{g})y_1w_\phi\)
\eeqn for a unique $\xi\in\CC$.
\end{itemize}
\end{thm}

\begin{proof}
Suppose first that $\r{rank}B_\phi=2$. Then $\r{null}B_\phi=0$. By the preceding lemma and the inequality $\r{rank}A_\phi\geq2$, we have $\r{rank}A_\phi=3$, so $\r{null}A_\phi=0$. The defining systems \eqref{F4.1}--\eqref{F4.2} therefore give $\F{g}^\phi=\F{p}$. Theorem \ref{T3.4} implies that $W(\phi)$ is irreducible.

Now suppose that $\r{rank}B_\phi=1$. The preceding lemma gives $\r{rank}A_\phi=2$. Since $(\b_1,\b_2)^\r{T}$ spans $\r{null}B_\phi$, the element $y_1$ spans a complement of $\F{p}_{\o{1}}$ in $\F{g}_{\o{1}}^\phi$. A direct computation gives
\beqn
[y_1,y_1]=2\(\b_1^2L_{-1}+2\b_1\b_2L_0+\b_2^2L_1\)=2y_0.
\eeqn
Because $\F{g}^\phi$ is a subalgebra, $y_0\in\F{g}_{\o{0}}^\phi$. Moreover, $y_0\neq0$, and $\dim\r{null}A_\phi=1$; hence $y_0$ spans a complement of $\F{p}_{\o{0}}$ in $\F{g}_{\o{0}}^\phi$.  The classification now follows from Theorem \ref{T3.10} and \eqref{F10}.
\end{proof}

\subsection{The complete spectrum-generating superalgebra}

The complete spectrum-generating superalgebra \cite{GMP} arose in the study of superstring theory as an infinite-dimensional symmetry algebra combining bosonic and fermionic modes. Its representation theory is therefore relevant to the algebraic organization of spectra in supersymmetric models.

Let $\ve\in\left\{0,\frac{1}{2}\right\}$. Following \cite{GMP}, let $\C{G}_\ve$ denote the complete spectrum-generating superalgebra, a complex Lie superalgebra with
\beqn
(\C{G}_\varepsilon)_{\overline{0}} = \bigoplus_{n\in\mathbb{Z}} \mathbb{C}L_n \oplus \bigoplus_{n\in\mathbb{Z}} \mathbb{C}I_n,\qquad
(\C{G}_\varepsilon)_{\overline{1}} = \bigoplus_{r\in \varepsilon+\mathbb{Z}} \mathbb{C}G_r \oplus \bigoplus_{r\in \varepsilon+\mathbb{Z}} \mathbb{C}J_r
\eeqn
whose nonzero brackets are
\begin{alignat*}{3}
&[L_m, L_n] = (m-n)L_{m+n},
&\qquad &[L_m, I_n] = -n\,I_{m+n},
&\qquad &[L_m, G_r] = \left(\frac{m}{2}-r\right)G_{m+r},
\\
&[L_m, J_r] = -\left(\frac{m}{2}+r\right)J_{m+r},
&\qquad &[G_r, G_s] = 2\,L_{r+s},
&\qquad &[G_r, J_s] = I_{r+s},
\\
&[I_n, G_r] = n\,J_{n+r}.
\end{alignat*}
Take 
\beqn
\F{p} = \bigoplus_{n\in\mathbb{Z}} \mathbb{C}I_n \oplus \bigoplus_{r\in \varepsilon+\mathbb{Z}} \mathbb{C}J_r,
\eeqn
Then $\F{p}$ is an ideal of $\C{G}_\ve$.

Let $\phi\colon\F{p}\rightarrow\CC$ be a Lie superalgebra homomorphism. Then
\beqn
\phi(J_r) = 0, \quad \forall r\in \ve+\ZZ.
\eeqn
Thus $\phi$ is determined by the scalars
\beqn
\phi(I_n),\quad n\in\mathbb{Z}.
\eeqn
Set
\beqn
S^{\phi} = \left\{~n \in \mathbb{Z}\middle|\phi(I_{n})\neq 0~\right\}.
\eeqn
We call $\phi$ {\bf finite} if $|S^{\phi}|<\infty$.

\begin{lem}
If $S^\phi=\{k\}$, then $\C{G}_\ve^\phi=\CC L_k\oplus\F{p}$.
\end{lem}

\begin{proof}
Let
\beqn
y_0=\sum_{m\in\ZZ}a_mL_m\in \C{G}_\ve^\phi,
\eeqn
where only finitely many coefficients $a_m$ are nonzero. Then
\beqn
0=\phi([y_0,I_n])=-n\sum_{m\in\ZZ}a_m\phi(I_{m+n})
=-n\phi(I_k)a_{k-n},\quad \forall n\in\ZZ.
\eeqn
Since $\phi(I_k)\neq0$, it follows that $a_m=0$ for $m\neq k$. A direct calculation shows that $L_k\in\C{G}_\ve^\phi$. Hence,
\beqn
(\C{G}^\phi_\ve)_{\o{0}}=\F{p}_{\o{0}}\oplus\CC L_k.
\eeqn	

Suppose that 
\beqn
y_1=\sum_{r\in \ve+\ZZ} b_r G_r\in \C{G}_\ve^\phi,
\eeqn
where only finitely many coefficients $b_r$ are nonzero. Then
\beqn
0=\phi([y_1,J_s])=\sum_{r\in\ve+\ZZ}b_r\phi(I_{r+s})
=\phi(I_k)b_{k-s},\quad\forall s\in\ve+\ZZ.
\eeqn
Therefore, $b_r=0$ for every $r\in\ve+\ZZ$, so $y_1=0$. Hence $(\C{G}^\phi_\ve)_{\o{1}}=\F{p}_{\o{1}}$.
\end{proof}

\begin{lem}
If $2\leq|S^\phi|<\infty$, then $\C{G}_\ve^\phi=\F{p}$.
\end{lem}

\begin{proof}
Set
\beqn
i=\min S^\phi,\quad j=\max S^\phi,
\eeqn
so that $i<j$.

Suppose that $(\C{G}_\ve^\phi)_{\o{0}}\neq\F{p}_{\o{0}}$. Then there exists a nonzero element
\beqn
y_0=\sum_{m\in \ZZ} a_m L_m\in (\C{G}_\varepsilon)^\phi_{\overline 0},
\eeqn
where only finitely many coefficients $a_m$ are nonzero. Let
\beqn
m_0=\min\{m\mid a_m\neq 0\},\quad m_1=\max\{m\mid a_m\neq 0\}.
\eeqn
We claim that $m_1=i$. Otherwise, set $n_0=i-m_1\neq0$. Then
\beqn
m+n_0<i,\quad \forall m<m_1.
\eeqn
For every $m<m_1$, the minimality of $i$ gives
\beqn
\phi(I_{m+n_0})=0,\quad\forall m<m_1.
\eeqn
Consequently,
\beqn
0=\phi([y_0,I_{i-m_1}])=-(i-m_1)a_{m_1}\phi(I_i).
\eeqn
This is a contradiction because
\beqn
i-m_1\neq 0,\quad a_{m_1}\neq 0,\quad \phi(I_i)\neq 0.
\eeqn
A similar argument, using $j=\max S^\phi$, gives $m_0=j$. This is impossible because $m_0\leq m_1$ whereas $j>i$. Therefore,
\beqn
(\C{G}_\ve^\phi)_{\o{0}}=\F{p}_{\o{0}}.
\eeqn

Next, suppose that $(\C{G}_\ve^\phi)_{\o{1}}\neq\F{p}_{\o{1}}$. Then there exists a nonzero element
\beqn
y_1=\sum_{r\in\ve+\ZZ} b_r G_r\in (\C{G}_\varepsilon)^\phi_{\overline 1},
\eeqn
where only finitely many coefficients $b_r$ are nonzero. Let
\beqn
r_1=\max\{r\mid b_r\neq 0\}.
\eeqn
Set $s_0=i-r_1$. Then
\beqn
r+s_0<i,\quad\forall r<r_1.
\eeqn
Hence,
\beqn
\phi(I_{r+s_0})=0,\quad\forall r<r_1.
\eeqn
Therefore,
\beqn
0=\phi([y_1,J_{i-r_1}])=b_{r_1}\phi(I_i),
\eeqn
a contradiction because $b_{r_1}\neq0$ and $\phi(I_i)\neq0$. Therefore,
\beqn
(\C{G}_\ve^\phi)_{\o{1}}=\F{p}_{\o{1}}.
\eeqn
This completes the proof.
\end{proof}

Theorems \ref{T3.4} and \ref{T3.7} immediately yield the following result.

\begin{thm}
Let $\phi\colon\F{p}\rightarrow\CC$ be a nonzero finite Lie superalgebra homomorphism.
\begin{itemize}
\item[(1)] If $|S^\phi|\geq2$, then $W(\phi)$ is irreducible.

\item[(2)] If $S^\phi=\{k\}$, then every irreducible quasi-Whittaker $\C{G}_\ve$-supermodule of type $\phi$ is isomorphic, up to parity shift, to the supermodule
\beqn
M_\xi(\phi):=W(\phi)/\C{U}(\C{G}_\ve)(L_k-\xi1)w_\phi
\eeqn
for a unique $\xi\in\CC$.
\end{itemize}
\end{thm}

\section* {Acknowledgment} 

 W. Gao is partially supported by  Natural Science Foundation of Henan (No.252300423524) and Natural Science Foundation of China (No. 12471024) .
 B. He is partially supported by  National Natural Science Foundation of China (Nos. 12001467 and 12471024) and Nanhu Scholars Program for Young Scholars of XYNU (No. 012109).
  S. Liu is partially supported by Natural Science Foundation of China (No. 12671041) and Nanhu Scholars Program of XYNU (No.012111). The authors would like to thank Professor Kaiming Zhao and Dr. Cunguang Cheng for their valuable suggestions.


	Wenting Gao, School of Mathematics and Statistics,Xinyang Normal University, Xinyang 464000, P. R. China. Email address: gaowentingxy@163.com

    Baiying He, School of Mathematics and Statistics, Xinyang Normal University, Xinyang 464000, P. R. China. Email address: heby621@nenu.edu.cn
    
    Shiyuan Liu, School of Mathematics and Statistics, Xinyang Normal University, Xinyang 464000, P. R. China. Email address: liushiyuanxy@163.com

    Jialin Xu, School of Mathematics and Statistics, Xinyang Normal University, Xinyang 464000, P. R. China. Email address: xujialin0908@163.com

\end{document}